\documentclass[11pt]{amsart}
\usepackage{amscd,amsfonts,amssymb,amsmath}
\usepackage[margin=3.2 cm]{geometry}
\usepackage{amsmath}
\usepackage{color}

\newtheorem{theorem}{Theorem}[section]
\newtheorem{cor}[theorem]{Corollary}
\newtheorem{lemma}[theorem]{Lemma}

\theoremstyle{definition}
\newtheorem{definition}[theorem]{Definition}
\newtheorem{example}[theorem]{Example}

\theoremstyle{remark}

\numberwithin{equation}{subsection}
\theoremstyle{plain}

\newtheorem{question}{Question}

\def\Z{\mathbb Z}

\def \P {\mathcal P}
\def \pw{{\rm pw}}
\def \F {F_n}

\newcommand{\secref}[1]{Section~\ref{#1}}
\newcommand{\thmref}[1]{Theorem~\ref{#1}}
\newcommand{\lemref}[1]{Lemma~\ref{#1}}

\newcommand{\corref}[1]{Corollary~\ref{#1}}

\numberwithin{equation}{section}

\begin{document}

\title[On  Palindromic Widths of Nilpotent and Wreathe Products]{On  Palindromic Widths of Nilpotent and Wreathe Products}

\author{Valeriy ~G.~Bardakov}
\address{Sobolev Institute of Mathematics, Novosibirsk State University, Novosibirsk 630090, Russia
and Laboratory of Quantum Topology, Chelyabinsk State University, Brat'ev Kashirinykh street 129, Chelyabinsk 454001, Russia}
\email{bardakov@math.nsc.ru}
\author{Oleg ~V.~Bryukhanov}
\address{Siberian University of Consumer Cooperatives, Novosibirsk 630087, Russia}
\email{bryuoleg@ngs.ru}
\author{Krishnendu Gongopadhyay}
\address{Department of Mathematical Sciences, Indian Institute of Science Education and Research (IISER) Mohali,
Knowledge City, Sector 81, S.A.S. Nagar, P.O. Manauli 140306, India}
\email{krishnendu@iisermohali.ac.in, krishnendug@gmail.com}
\subjclass[2000]{Primary 20F16; Secondary 20F65, 20F19, 20E22}
\keywords{palindromic width;  commutator width;  wreath products; nilpotent product}

\thanks{The authors gratefully acknowledge the support of the Indo-Russian DST-RFBR project grant DST/INT/RFBR/P-137}
\thanks{Bardakov is partially supported by Laboratory of Quantum Topology of Chelyabinsk State University (Russian Federation government grant 14.Z50.31.0020) }

\date{\today}


\begin{abstract}
We  prove  that nilpotent product of a set groups $A_{1},\dots, A_{s}$ has finite palindromic width if and only if  palindromic widths of $A_{i}, i=1,\dots, s,$ are finite. We give a new proof that the commutator width of $F_n \wr K$ is infinite, where $F_n$ is a free group of rank $n \geq 2$ and $K$ a finite group. This result, combining with a result of Fink \cite{f1} gives examples of groups with infinite commutator width but finite palindromic width with respect to some generating set.
\end{abstract}
\maketitle

\section{Introduction}
Let $G$ be a group with a set of generators $X$. A reduced word in the alphabet $X^{\pm 1}$
is a \emph{palindrome} if it reads the same forwards and backwards.
The palindromic length $l_{\mathcal P}(g)$
of an element $g$ in $G$ is the minimum number $k$ such that $g$ can be expressed as a product of $k$
palindromes. The \emph{palindromic width} of $G$ with respect to $X$ is defined to
be ${\rm pw}(G, X) = \underset{g \in G}{\sup} \ l_{\mathcal{P}}(g)$. When there is no confusion about the underlying generating set $X$, we simply denote the palindromic width with respect to $X$ by ${\rm pw}(G)$. Palindromic width of free groups was investigated by Bardakov, Shpilrain and Tolstykh \cite{BST} who proved that the palindromic width of a non-abelian free group is infinite. This result was generalized by Bardakov and Tolstykh \cite{BT} who proved that almost all free products have infinite palindromic width; the only exception is given by the free product of two cyclic groups of order two, when the palindromic width is two. Piggot \cite{P} studied the relationship between primitive words and palindromes in free groups of rank two.  Recently, there have been a series of work that aim to understand palindromic widths in several other classes of groups.  In a series of papers \cite{BG, BG2, BG3} the authors have proved finiteness of palindromic widths of finitely generated free nilpotent groups and certain solvable groups. Some bounds of the widths were also obtained in many cases.  Palindromic widths of wreath products and Grigorchuk groups have been investigated by
Fink \cite{f1, f3}. Riley and Sale have investigated palindromic widths in certain wreath products and solvable groups \cite{RS} using finitely supported functions from $\Z^r$ to the given group. Fink and Thom \cite{f2} have studied palindromic widths in simple groups.

 \medskip For $g$, $h$ in $G$, the \emph{commutator} of $g$ and $h$ is defined as $[g,
h]=g^{-1} h^{-1}gh$. If $\mathcal{C}$ is the set of commutators in some group
$G$ then the commutator  subgroup $G'$ is generated by $\mathcal{C}$.
The \emph{commutator length} $l_{\mathcal{C}}(g)$ of an element $g \in G'$ is the minimal number $k$ such that $g$ can be expressed as a product of $k$ commutators.  The \emph{commutator width} of $G$ is defined by $\underset{g \in G'}{\sup} \ l_{\mathcal{\mathcal C}}(g)$ and is
denoted by ${\rm cw}(G)$. Commutator width in groups have been studied by several authors, for eg. see \cite{AM, B, R}.  It is well known \cite{R} that the commutator width of a
free non-abelian group is infinite, but the commutator width of a finitely
generated nilpotent group is finite. An algorithm of the
computation of the commutator length in free non-abelian groups can be found in \cite{B}.

\medskip In this paper, we address two problems related to palindromic widths in groups. First,  we investigate the palindromic widths of nilpotent products of groups. Bardakov and Gongopadhyay \cite{BG, BG2} proved that the palindromic width of a free nilpotent group is finite. We extend this result for nilpotent products of groups.  Recall that the concept of nilpotent products arises from the work of Golovin \cite{G}, also see \cite{G1, G2}. The construction of nilpotent products also follow from the construction of so called verbal  products of groups, see Moran  \cite{SM1,SM2,SM3}.  This construction appeared to answer a question by Kurosh who asked whether there are any other products, other than the free and the direct products of groups, which also have the following properties:
\begin{enumerate}
\item[(a)] the products are commutative;
\item[(b)] the products are associative;
\item[(c)] there are generating subgroups of the products, that is, given a product $G$, it has a subgroup $S$ such that $G=\langle S \rangle$.
\item[(d)] the intersection of a given one of these subgroups with  the normal subgroups generated  by the rest of these  subgroups is the identity.
\end{enumerate}
The construction of the nilpotent products generalize the free and direct products of groups. Every nilpotent group is a quotient of a nilpotent product. In this paper we prove that any nilpotent product of set of groups $A_{1} = \langle X_1 \rangle, \dots, A_{s} = \langle X_s \rangle$ has finite palindromic width with respect to the generating subset $\bigcup\limits_{i=1}^{s}X_{i}$ if and only if ${\rm pw}(A_{i}, X_{i}), i=1,\dots, s,$ are finite. We prove this result in \secref{np}.

\medskip The later part of the paper addresses relationship between commutator and palindromic widths. In \cite[Problem 2]{BG} we ask for the relationship between commutator and palindromic widths of groups and  provided some relationship between them in \cite{BG3}. We know that the palindromic and commutator widths of the free non-abelian group $F_n = \langle X \rangle$ is infinite. On the other side, Akhavan-Malayeri \cite{AM} proved that the commutator width of wreath product $F_n \wr \mathbb{Z}^m$ is finite. Analogous result for the palindromic width has been proved by  Fink \cite{f1}.
Recently, Fink and Thom \cite{f2} have shown that there exists finitely generated simple groups having infinite commutator width but finite palindromic width with respect to some generating set.

On the other hand, for $K$ a finite group,  Fink \cite{f1} has proved that there exists a generating set $S$ such that $\pw(F_n \wr K, S)$ is finite.  We show that for $K$ a finite group, the commutator width of $F_n \wr K$, $n \geq 2$, is infinite.  This result is already known from the work of Nikolov \cite{n}. However, we give a different proof of this result using  standard ideas  of constructing a quasi-homomorphism, for eg. see \cite{R, BST}.  Thus $F_n \wr K$ provides first example of  a non-simple group that has infinite commutator width but finite palindromic width with respect to some generating set.

For completeness of this paper,  we demonstrate Fink's ideas by considering the simple example of $F_2 \wr S_3$  where $S_3$ is the symmetric group of three letters.  Using Fink's ideas we show that the palindromic width of this group is at most 20 with respect to the canonical set of generators. This is actually an improvement of the bound of Fink that turns out to be 40 in this case. In \cite{BG3}, we proved that the palindromic width of finite extension of a group with finite palindromic width is finite. But here we see that $F_2 \wr S_3$ is a finite extension of the group $F_2^6$ that has infinite palindromic width. So, this gives example of a group that has infinite palindromic width but finite extension has finite palindromic width.

\section{Two approaches to Palindromes}
There are two notions of palindromic words in a group that have been implicit in recent literature. In the following we compare these notions.

\medskip {\bf First notion. } Let $A$ be an alphabet. A \emph{word} in these alphabets is a sequence of letters $u=a_1 a_2 \dots a_k, ~ a_i \in A$. An empty word is denoted by 1. A word $u=a_1 a_2 \dots a_k$ is equal to a word $v=b_1 b_2 \dots b_l$ if $k=l$ and $a_1=b_1, ~ a_{2}=b_{2}, \dots, a_k=b_k$. Let $\bar u$ denote the reverse word $\bar u= a_k a_{k-1} \dots a_1$. We say that $u$ is a \emph{palindrome} with respect to $A$ if $u=\bar u$, i.e. $a_1=a_{k}, a_2=a_{k-1}, \dots, a_i=a_{k-i+1}, \dots$.

Let $G$ be a group with a generating set $A$. We assume that $A$ is symmetric, i.e. $A=A^{-1}$. Let $A^{\ast}$ be the free monoid over the alphabet $A$. There is a homomorphism $\rho: A^{\ast} \to G$ that sends every word $a_1 a_2 \dots a_k$ to some element in $G$. Evidently, for an element $g \in G$, there are a lot of words $u$ in $A^{\ast}$ such that $\rho(u)=g$; we denote this set of elements by $\rho^{-1}(g)$. An element $g$ in $G$ is a \emph{palindrome} (or, \emph{word palindrome}) with respect to $A$ if there is a palindrome in the set $\rho^{-1} (g) \subset A^{\ast}$. In this paper we shall follow this notion.

\medskip {\bf Second notion.} This notion has been used by Fink \cite{f1}. We call an element $g$ in $G$ a \emph{group palindrome} if it is represented by a word $g=a_1 a_2 \dots a_k$, $a_i \in A$, such that $\bar g$ represent the same element in $G$. Evidently if $g$ is a word palindrome then it is a group palindrome. But the converse is not true.
  In \cite[Lemma 3.1]{f1} Fink has implicitly used this notion to assert that $g=\bar g$ in any Abelian group. Using the first notion, it is not true that $g=\bar g$ in an Abelian group, see the following example.
\begin{example}
Let $G=\Z^{\oplus n}$ be the free abelian group of rank $n$. Let $\{a_1, a_2, \dots, a_n \}$ be a basis of $G$. Using the first definition we proved that $\pw(G, A)=n$, see \cite{BG}. Using the second definition we see that every element of $g$ is a palindrome, since $a_1^{\alpha_1} \dots a_n ^{\alpha_n}=a_n^{\alpha_n} \dots a_1^{\alpha_1}$ and hence for any $g=a_1^{\alpha_1} \dots a_n ^{\alpha_n}$ in $G$ we have $g =\bar g$. Hence by the second definition palindromic width of $G$ is $1$.
\end{example}

If we denote by $\P_G$ the set of group palindromes and $\P_A$ the set of word palindromes in $A$, then $\P_A \subset \P_G$. Also, the palindromic width with respect to the group palindromes does not exceed $\pw(G, A)$. If $F$ is a free group with basis $X$, then $\P_X=\P_F$.

It would be interesting to compare the results of Fink \cite{f1, f2}  from the above point of views.

\section{Palindromic width of nilpotent product of groups }\label{np}

 Bardakov and Gongopadhyay \cite{BG, BG2} investigated the palindromic width of the free nilpotent groups. In this section the palindromic width of nilpotent product of groups is investigated.  Recall the construction of the nilpotent product of groups. This construction was defined in the paper of Golovin \cite{G}.

Let $A\ast B$ be the free product of some groups $A$ and $B$. Cartesian subgroup  and lower central series  of $A\ast B$ are denoted  by $[A, B]$ and $\gamma_{n}(A\ast B), n=1, 2,\dots,$ respectively. 
The $n$-th nilpotent product $G_n=A(n)B,$ $n\geq2,$ is defined as the quotient \hbox{$A\ast B/[A, B]\cap\gamma_{n+1}(A\ast B)$}. It is clear that $A(1)B=A\times B$ is the direct sum.

\medskip Let us list some common properties of the nilpotent product $G_n$ from \cite{G}.
\begin{enumerate}
\item $A,B\leq G_n,$ $ A^{G_n}\cap B=1,$ $ A\cap B^{G_n}=1$ and $G_n=\langle A, B\rangle$.
\item Any element $g\in G_n$ can be  uniquely written as a product $a \cdot b \cdot w(g)$,  where $a\in A,$ $b\in B,$ $w(g)\in[A, B]$.
\item Given normal subgroups $A_{0}\unlhd A$, $B_{0}\unlhd B$ so that $A_{0}, B_{0}\unlhd G_n$, let $\overline A=A/A_0$, $\overline B=B/B_0$. Then the group homomorphisms $A\rightarrow A/A_{0}$ and $B\rightarrow B/B_{0}$ are extended to the homomorphism of  the nilpotent products \hbox{$\Phi: G_n\rightarrow {\overline A}(n) \overline B$} with $\mathrm{Ker }~\Phi=A_{0}B_{0}$ where $A_{0}B_{0}\cap[A, B]=1$.
\item $[g_{1}, g_{2},\dots, g_{n+1}]=1$ where $g_{i}\in A \cup B$.
\item $(A(n)B)(n)C=A(n)(B(n)C)=(A(n)C)(n)B$.
\end{enumerate}

\medskip For subsets $X$,  $Y$ in $G_n$, let $C_{X}(Y)$ denote the centralizer of $Y$ in $X$.
\begin{lemma}\label{npl}
Given nilpotent product $G_n$, the following holds.
\begin{itemize}
\item[(i)] $\gamma_{n}(A)\leq C_{A}(B)$, $\gamma_{n}(B)\leq C_{B}(A)$.
\item[(ii)] $C_{A}(B), C_{B}(A)\unlhd G_n$.
\item[(iii)] $C_{A}(B)C_{B}(A)\cap[A, B]=1$.
\end{itemize}\end{lemma}
\begin{proof}
(i) It follows from the inclusions $[\gamma_{n}(A), B], [\gamma_{n}(B), A]\subseteq [A, B]\cap\gamma_{n+1}(A\ast B)=1.$

(ii) Let us prove    that the subgroups $C_{A}(B)$ and $C_{B}(A)$ are characteristic subgroups in $A$ and $B$ accordingly. Note that any   pair  of automorphisms $\varphi\in\mathrm{Aut}A$ and $\psi\in\mathrm{Aut}B$ induce the automorphism $\zeta\in\mathrm{Aut}(A\ast B)$. Given commutator
$[a, b]\in[A,B]\cap\gamma_{n+1}(A\ast B)$ we have,   $[a^{\varphi}, b^{\psi}]=[a, b]^{\zeta}\in([A,B]\cap\gamma_{n+1}(A\ast B))^{\zeta}=[A,B]\cap\gamma_{n+1}(A\ast B)$. It follows that the subgroups $C_{A}(B)$ and $C_{B}(A)$ are characteristic subgroups in $A$ and $B$ accordingly and so $C_{A}(B)\unlhd A$ and $C_{B}(A)\unlhd B$. This implies that $C_{A}(B), C_{B}(A)\unlhd G_n$.

(iii) This statement is clear following the  proved statement (ii) and property $(3)$ of the list above.
\end{proof}

It is clear that groups $\overline{A}=A/C_{A}(B)$, $\overline{B}=B/C_{B}(A)$ are nilpotent of step $\leq n-1$ and nilpotent product ${\overline G}_n=\overline{A}(n)\overline{B}$ is nilpotent of step $\leq n$. Let $\overline{X}$ and $\overline{Y}$ be generating sets of groups $\overline{A}$ and $\overline{B}$ with $|\overline{X}|=m_{A}$, $|\overline{Y}|=m_{B}$.

\begin{theorem}\label{Wnp}
Let $A$ and $B$ are finitely generated groups with generating set $X$ and $Y$ respectively. Given the nilpotent product $G_n = A(n)B$, the following holds.
\begin{itemize}

\item[(i)] {$\max\{{\rm pw}(A, X), {\rm pw}(B, Y)\}\leq {\rm pw}(G_n, X\cup Y)\leq {\rm pw}(A, X)+{\rm pw}(B, Y)+3(m_{A}+m_{B})$.}

\medskip \item[(ii)] If $A=C_{A}(B)$ or $B=C_{B}(A)$ then
$$\max\{{\rm pw}(A, X), {\rm pw}(B, Y)\}\leq {\rm pw}(G_n, X\cup Y)\leq {\rm pw}(A, X)+{\rm pw}(B, Y).$$
\end{itemize} \end{theorem}
\begin{proof}  (i) \lemref{npl} and property $(3)$ imply   that there is a homomorphism $G_n\rightarrow\overline{G}_n$ with kernel $C_{A}(B)C_{B}(A)\unlhd G_n$ and $C_{A}(B)C_{B}(A)\cap[A, B]=1$. Let $X$ and $Y$ be generating sets of groups $A$ and $B$, and let $\overline{X}$ and $\overline{Y}$ be  the corresponding generating sets of groups $\overline{A}$ and $\overline{B}$.

Next, let $X_{0}\subseteq X$ and $Y_{0}\subseteq Y$ be sets of  representatives of elements of   $\overline{X}$ and $\overline{Y}$ respectively. Let $\overline{g}\in\overline{G}_n$ be homomorphic image of $g\in G_n$, then $\overline{g}=\overline{p}_{1}(\overline{X}, \overline{Y})\cdots\overline{p}_{s}(\overline{X}, \overline{Y})$ and $$g=\overline{p}_{1}(X_{0}, Y_{0})\cdots\overline{p}_{s}(X_{0}, Y_{0})\cdot ab$$ for some palindromes
$\overline{p}_{i}(\overline{X}, \overline{Y})$,
$a \in C_{A}(B)$, $b \in C_{B}(A)$.
Hence, $${\rm pw}(G_n, X\cup Y)\leq{ \rm pw}(A, X)+{\rm pw}(B, Y)+{\rm pw}(\overline{G}_n, \overline{X}\cup\overline{Y}).$$
The nilpotent group $\overline{G}_n$ is the homomorphic image of some free nilpotent group ${\rm N}_{m,s}$ with rank $m=m_{A}+m_{B}$ and step $s\geq n$. Therefore, ${\rm pw}(\overline{G}_n, \overline{X}\cup \overline{Y})\leq\rm{pw}(N_{m,s})\leq3(m_{A}+m_{B})$. The last inequality  follows from \cite[Theorem 1.1]{BG}. Further, $A\times B$ is homomorphic image of $G_n$ and groups $A$ and $B$ are homomorphic images of $A\times B$, so
$$\max\{{\rm pw}(A, X), {\rm pw}(B, Y)\}\leq{\rm pw}(A\times B, X\cup Y)\leq{\rm pw}(G_n, X\cup Y).$$
Finally, $$\max\{\rm{pw}(A, X), \rm{pw}(B, Y)\}\leq{ \rm pw}(G_n, X\cup Y)\leq\rm{pw}(A, X)+\rm{pw}(B, Y)+3(m_{A}+m_{B}).$$

(ii) If $A=C_{A}(B)$ or $B=C_{B}(A)$ then $G_n=A\times B$ and so $$\max\{{\rm pw}(A, X), {\rm pw}(B, Y)\}\leq{ \rm pw}(G_n, X\cup Y)\leq\rm{pw}(A, X)+\rm{pw}(B, Y).$$

This proves the lemma.
\end{proof}

Using properties (1) -- (5), one can define  nilpotent products $$(n)\{A_{1},\dots,A_{s}\}=(\dots(A_{1}(n)A_{2})(n)\dots)(n)A_{s}$$
of a set of groups $A_1, \ldots, A_s$ inductively that also preserve properties (1) -- (5), see \cite{G}. Let $A_{i}=\langle X_{i}\rangle$, $C_{i}=C_{A_{i}}( A_1 \cup \ldots \cup A_{i-1} \cup \widehat{A_{i}} \cup A_{i+1} \cup \ldots \cup A_s)$, $\overline{A}_{i}=A_{i}/C_{i}$,
$\overline{A}_{i}=\langle\overline{X}_{i}\rangle$, where $\overline{X}_{i}$ is the homomorphic image of $X_{i}$ and, $|\overline{X}_{i}|=m_{i}$ for all $i=1,\dots,s$. It is clear that $C_{i}=\bigcap\limits_{k\neq i}C_{A_{i}}(A_{k})$. Then, it is followed from Lemma \ref{npl} that $\gamma_{n}(A_{i})\leq C_{i}$ and $C_{i}\trianglelefteq(n)\{A_{1},\dots,A_{s}\}$. Hence, $(n)\{A_{1},\dots,A_{s}\}/C_{1}\cdots C_{s}=(n)\{A_{1}/C_{1},\dots,A_{s}/C_{s}\}$ is a nilpotent group with generating set $\bigcup\limits_{i=1}^{s}\overline{X}_{i}$ that consists $\sum\limits_{i=1}^{s}m_{i}$ elements. Finally, if $A_{k}=C_{k},$  then   $(n)\{A_{1},\dots,A_{s}\}= A_{k}\times(n)\{A_{1},\dots,\widehat{A_{k}},\dots,A_{s}\}$.  Further,
arguments same as in theorem \ref{Wnp}(i, ii) prove the followed statement.
\begin{theorem}\label{Wnps}
Given the nilpotent product $G_{n}=(n)\{A_{1},\dots,A_{s}\}$, the following holds.
\begin{itemize}
 \item[(i)]  $\max\limits_{i=1,\dots,s}\{{\rm pw}(A_{i}, X_{i})\}\leq
 {\rm pw}(G_{n},\bigcup\limits_{i=1}^{s}X_{i})\leq\sum_{i=1}^{s}
    {\rm pw}(A_{i}, X_{i})+3\sum_{i=1}^{s}m_{i}.$
    \item[(ii)] If $A_{k}=C_{k},$  $G_{n}(\widehat{A_{k}})=(n)\{A_{1},\dots,\widehat{A_{k}},\dots,A_{s}\}$ then
$$\max\{{\rm pw}(A_{k}, X_{k}), {\rm pw}(G_{n}(\widehat{A_{k}}), \bigcup\limits_{i\neq k}X_{i})\leq {\rm pw}(G_{n},\bigcup\limits_{i=1}^{s}X_{i})\leq {\rm pw}(A_{k}, X_{k})+{\rm pw}(G_{n}(\widehat{A_{k}}), \bigcup\limits_{i\neq k}X_{i}).$$
\end{itemize}
\end{theorem}

\medskip The last theorem means that nilpotent product $(n)\{A_{1},\dots,A_{s}\}$ has finite palindromic width with respect to the generating set $\bigcup\limits_{i=1}^{s}X_{i}$ if and only if all palindromic widths ${\rm pw}(A_{i}, X_{i}),$  $i=1,\dots, s,$ are finite.

\section{On Commutator and Palindromic Widths}

\begin{definition}
Let $G$ be a group. We say that a map $f:G \to \Z$ is a quasi-homomorphism if there is some constant $c$ such that for every $g, h \in G$, we have
$$|f(gh)-f(g)-f(h)|\leq c,$$
\end{definition}

Define a function $tr: \Z \to \{-1, 0, 1\}$ by
$$
tr(m) =
\left\{\begin{array}{lll}
-1 & ~\mbox{if}~m \equiv -1 \mod 3, \\
0 & ~\mbox{if}~~m \equiv 0 ~ ~\mod 3, \\
1 & ~ \mbox{if}~~ m \equiv 1 ~ ~ \mod 3.
\end{array}
\right.
$$
The following is easy to prove.
\begin{lemma}\label{lem1}
If $m, n \in \Z$, then
\begin{enumerate}
\item $tr(m)+tr(n)-3 \leq tr(m+n) \leq tr(m) + tr(n) +3$.
\item $tr(-m)=-tr(m)$.
\end{enumerate}
\end{lemma}
Let  $F_n=\langle x_1, \ldots, x_n \rangle$, $n \geq 2$ be the free group of rank $n$.
Let $w=x_{i_1}^{\alpha_1} x_{i_2}^{\alpha_2} \ldots x_{i_t}^{\alpha_t}$ be a reduced word in $\F$. Define $ql: \F \to \Z$ by
$$ql(w)=\sum_{i=1}^t tr(\alpha_i).$$
The following properties of $ql$ follows from the work of Bardakov \cite{B} who uses ideas of Rhemtulla \cite{R} to prove these properties.
\begin{lemma}\label{ql1}
For $f, g \in \F$,
\begin{itemize}
\item[(i)] $ql(f) + ql(g)-3 \leq ql(fg) \leq ql(f) + ql(g) +3$.
\item[(ii)] $ql(g^{-1})=-ql(g)$.
\item[(iii)] $-9 \leq ql([f,g]) \leq 9$.
\end{itemize}
\begin{proof}
(i) Follows from (1) of \lemref{lem1}. We prove (ii).

Let $g=x_{i_1}^{\alpha_1} x_{i_2}^{\alpha_2} \ldots x_{i_t}^{\alpha_t}$. Then
$$ql(g)=tr(\alpha_1)+ tr(\alpha_2)+ \dots+ tr(\alpha_t).$$
Take $g^{-1} =x_{i_t}^{-\alpha_t}x_{i_{t-1}}^{\alpha_{t-1}} \ldots x_{i_1}^{\alpha_1}$. Then
$$ql(g^{-1})=tr(-\alpha_t)+\dots+tr(-\alpha_1)=-ql(g).$$

(iii) follows from (i) and (ii).
\end{proof}
\end{lemma}

\begin{theorem}\label{cw}
Let $K$ be a finite group and $\F$ be a free group of rank $n \geq 2$. The commutator width of $F_n \wr K$ is infinite.
\end{theorem}

Let $G=F_n \wr K$.  Let  $K=\{k_1, \ldots, k_l\}$. To prove the above theorem, we shall use $ql$ to define a quasi-homomorphism $\Delta: G \to \Z$ on $G$. Note that $G=P \rtimes K$, where $P$ is the direct product $P=\prod_{i=1}^l F_{k_i}$, each $F_{k_i}$ is an isomorphic copy of $\F$. Thus, every element $g \in G$ has a form
$g=(f_{k_1}, f_{k_2}, \ldots, f_{k_l}) k$, where $k \in K$, $f_{k_i} \in \F$. The group $K$ acts on $P$ by the natural action:
$(f_{k_1}, \ldots, f_{k_l})^k=(f_{k^{-1}k_1}, \ldots, f_{k^{-1}k_l})$, and this further induces an action of the symmetric group $S_l$ on $P$:
$(f_{k^{-1}k_1}, \ldots, f_{k^{-1}k_l})=(f_{k_{\sigma(1)}}, \ldots, f_{k_{\sigma(l)}})$, for some $\sigma \in S_l$.

Define a quasi-homomorphism $\Delta: G \to \Z$ by
$$\hbox{For }g=(f_{k_1}, f_{k_2}, \ldots, f_{k_l}) k \in G \hbox{ let, } \Delta(g)=\sum_{i=1}^l ql(f_{k_i}).$$
It follows from the above lemmas using standard methods that:
\begin{lemma}\label{gh2} Let $g, h \in G$. Then the following holds.
\begin{enumerate}
\item $|\Delta (gh)- \Delta(g) - \Delta (h)|\leq 3l$.
\item $|\Delta(g) + \Delta(g^{-1} )|\leq 3l$.
\item $|\Delta([g, h])|\leq 15l$.
\end{enumerate}
\end{lemma}
\begin{proof}
Let $g=(f_{k_1}, \dots, f_{k_l})k$, $h=(f'_{k_1}, \dots, f'_{k_l})k'$. Then
$$gh=(f_{k_1}f'_{k k_1}, \dots, f_{k_l} f'_{k k_l})kk', $$
and we have,
\begin{eqnarray*}
|\Delta(gh)-\Delta(g)-\Delta(h)| & = & | \sum_{i=1}^l ql(f_{k_i} f'_{k k_i} )- \sum_{i=1}^l ql(f_{k_i})-\sum_{i=1}^l ql(f'_{k_i}) | \\
&\leq& \sum_{i=1}^l | ql(f_{k_i} f'_{k k_i} )-ql(f_{k_i})-ql(f'_{k_i})|\\
&\leq & 3l ~~~\hbox{(by \lemref{ql1}).}
\end{eqnarray*}

For (2), take $h=g^{-1}$ in (1). Noting that $\Delta(1)=0$ the inequality follows.

(3) follows from (1) and (2):
$$|\Delta([g, h])-\Delta(g^{-1}) -\Delta(h^{-1} )-\Delta(g)-\Delta(h)|\leq 9l.$$
This implies, $|\Delta([g, h])|\leq 15l$.
\end{proof}
\begin{cor}\label{cor1}
If $g$ in $G$ is a product of $m$ commutators, then $|\Delta(g)|\leq 3l(6m-1).$
\end{cor}
\begin{proof}
By repeated application of the above lemma,

\medskip
$|\Delta(g)|\leq 15lm+3l(m-1)=3l(6m-1)$.
\end{proof}
\subsection{Proof of \thmref{cw}}
 We consider the sequence $q_j=(a_j, 1, \ldots, 1)$ in $G$, where $a_j=x_2^{-3 j} x_1^{-3j} (x_2 x_1)^{3j} \in  F_{k_1} \subset G$. Then
$$\Delta(q_j)=ql(x_2^{-3 j} x_1^{-3j} (x_2 x_1)^{3j})=6j.$$
So,
$$\Delta(q_1) < \Delta(q_2) < \ldots <\Delta(q_j) < \ldots$$
Thus the sequence of elements ${q_j}$ is unbounded on the $\Delta$-values. On the other hand, by \corref{cor1}, every element of $g$ that is a product of bounded number of commutators, must have a finite $\Delta$-value. This shows that $G$ can not have a bounded  commutator width.

\subsection{Palindromic width of $F_2 \wr S_3$} Let $S_3$ be the symmetric group of three symbols. It is isomorphic to the Dihedral group
$$S_3=\langle s_1, s_2 \ | \ s_1^2=s_2^2=(s_1 s_2)^3=1 \rangle.$$
Following Fink \cite{f1}, we add a new generator $c=s_1 s_2$  to the generating set of $S_3$:
$$S_3=\langle s_1, s_2, c \ | \ s_1^2=s_2^2=(s_1 s_2)^3=1, ~ c=s_1 s_2 \rangle.$$
Let $G=F_2 \wr S_3$. Express an element $f \in G$ in the form:
$$f=(f_1, f_2, f_3, f_4, f_5, f_6)s,$$
where $f_i=x^{\alpha_i} y^{\beta_i} g_i, ~ \alpha_i, \beta_i \in \Z, ~ g_i \in F_2', ~ s \in S_3$.  Let $F_2^6$ denote the direct product of six copies of $F_2$.
\begin{lemma}
 Every element of $F_{2}^{6}$ can be written as a product of 18 palindromes with respect to the generating set $\{x, y, s_{1}, s_{2}, c\}$.
\end{lemma}
\begin{proof}
Note that the element
$$(x^{\alpha_1} y^{\beta_1}, x^{\alpha_2} y^{\beta_2}, \ldots, x^{\alpha_6} y^{\beta_6})
=(x^{\alpha_1}, x^{\alpha_2}, \ldots, x^{\alpha_6})(y^{\beta_1}, y^{\beta_2}, \ldots, y^{\beta_6}),$$
is a product of twelve palindromes.

First, we consider the case $g$ is a commutator to demonstrate the general method.
Let $g=([x, y], 1, 1, 1, 1, 1)$. Let $r=c(s_1s_2)^2$. Then $r=1$ in $S_3$. But $\bar r =(s_2 s_1)^2 c\neq 1$.
We can write $[x, y]=rx^{-1} r^{-1} y^{-1} r x r^{-1} y $. Note that
$\overline{[x, y]}= y \bar r^{-1} x \bar r y^{-1} \bar r^{-1} x^{-1} \bar r$. Then
\begin{eqnarray*}\bar g&=&(y \bar r^{-1} x \bar r y^{-1} \bar r^{-1} x^{-1} \bar r, 1, 1, 1, 1, 1)\\
&=&(yy^{-1}, 1,1, 1,1,1)(1,1,xx^{-1},1,1,1)=1.\end{eqnarray*}
Then $g=g \bar g$ is a palindrome.

Consider $w \in F_2 '$. Then $w=x^{a_1} y^{b_1} \ldots x^{a_l} y^{b_l}, \sum a_i=\sum b_i=0$.
We re-write $w$ as $w=rx^{a_1} r^{-1} y^{b_1} r x^{a_2} r^{-1} y^{b_2} \ldots r x^{a_l}r^{-1} y^{b_l}$. Then $\bar w=1$.

Consider $g=(g_1, \ldots, g_6) \in (F_2')^6$. Using the above expression, arbitrary element $g \in (F_2')^6$ can be written such that $\bar g=1$. Hence we can write
$$g=(g_1, g_2, g_3, g_4, g_5, g_6) (\bar g_1, \bar g_1, \bar g_3, \bar g_4, \bar g_5, \bar g_6),$$ where each $\bar g_i=1$. Consequently, $g$ is a product of six palindromes.

Thus an element $(f_1, \cdots, f_6) \in F_2^6$ can be written as a product of 18 palindromes with respect to the generating set $\{x, y, s_{1}, s_{2}, c\}$.
\end{proof}

Further, $$S_{3}=\{1, s_{1}, s_{2}, s_{1}s_{2}, s_{2}s_{1}, s_{1}s_{2}s_{1}\}=\{1, s_{1}, s_{2}, c, c^{-1}, cs_{1}\}$$ and so every element of $S_{3}$ can be written as a product of not more then two palindromes. So, every element of $G$ can be written as a product of $\leq 20$ palindromes.
From Fink's result \cite[Theorem 4.7]{f1}, it follows that $\pw(F_2 \wr S_3, S) \leq 40$. So, our construction above improves the bound of the palindromic width of $F_2 \wr S_3$.

\end{document}